\documentclass[preprint,12pt]{elsarticle}
\usepackage{amsthm,amsmath,amssymb}
\usepackage{graphicx}
\usepackage[colorlinks=true,citecolor=black,linkcolor=black,urlcolor=blue]{hyperref}
\usepackage{mathrsfs}

\theoremstyle{plain}
\newtheorem{theorem}{Theorem}
\newtheorem{lemma}[theorem]{Lemma}
\newtheorem{corollary}[theorem]{Corollary}

\theoremstyle{definition}
\newtheorem{definition}[theorem]{Definition}

\theoremstyle{remark}

\begin{document}

\title{On the polymatroid Tutte polynomial}
\author{Xiaxia Guan\\
\small School of Mathematical Sciences\\[-0.8ex]
\small Xiamen University\\[-0.8ex]
\small P. R. China\\
{\small\tt Email:gxx0544@126.com}\\
Weiling Yang\\
\small School of Mathematical Sciences\\[-0.8ex]
\small Xiamen University\\[-0.8ex]
\small P. R. China\\
{\small\tt Email:ywlxmu@163.com}\\
Xian'an Jin\footnote{Corresponding author}\\
\small School of Mathematical Sciences\\[-0.8ex]
\small Xiamen University\\[-0.8ex]
\small P. R. China\\
{\small\tt Email:xajin@xmu.edu.cn}
}
\begin{abstract}
The Tutte polynomial is a well-studied invariant of  matroids. The polymatroid Tutte polynomial $\mathcal{J}_{P}(x,y)$, introduced by Bernardi et al., is an extension of the classical Tutte polynomial from matroids to polymatroids $P$. In this paper, we first prove that $\mathcal{J}_{P}(x,t)$ and $\mathcal{J}_{P}(t,y)$ are interpolating for any fixed real number $t\geq 1$, and then we study the coefficients of high-order terms in $\mathcal{J}_{P}(x,1)$ and $\mathcal{J}_{P}(1,y)$. These results generalize results on interior and exterior polynomials of hypergraphs.
\end{abstract}

\begin{keyword}
Interpolating behavior\sep Tutte polynomial\sep Polymatroid\sep High-order term
\MSC 05C31\sep 05B35\sep 05C65
\end{keyword}
\maketitle

\section{Introduction}
\noindent

The Tutte polynomial \cite{Tutte} is an important and well-studied branch of graph and matroid theory, having wide applications in statistical physics  and knot theory and so on.  As a generalization of the valuations $T_G(x,1)$ and $T_G(1,y)$ in the Tutte polynomial $T_G(x,y)$ of graphs $G$ to hypergraphs, K\'{a}lm\'{a}n \cite{Kalman1} introduced the interior polynomial $I_{\mathcal{H}}(x)$ and the exterior polynomial $X_{\mathcal{H}}(y)$ of hypergraphs $\mathcal{H}$, respectively.  Later, K\'{a}lm\'{a}n and Postnikov \cite{Kalman2,Kalman3} established a relation between the top of the HOMFLY polynomial \cite{Freyd,Hom2}, which is a generalization of the celebrated Jones polynomial \cite{Jones} in knot theory, of any special alternating link and the interior polynomial of the pair of hypergraphs coming from the Seifert graph (which is a bipartite graph) of the link.

Polymatroids, firstly introduced by Edmonds \cite{Edmonds}, are a generalization of matroids and an abstraction of hypergraphs. In this paper, we only consider integer polymatroids, and simply call them polymatroids. In \cite{Bernardi}, Bernardi et al. proposed the polymatroid Tutte polynomial $\mathcal{J}_{P}(x,y)$, which is an extension of the classical Tutte polynomial from matroids to polymatroids. It is easy to see that for a hypergraphical polymatroid $P$ with ground set $[n]=\{1,\ldots,n\}$, the polynomial $\mathcal{J}_{P}(x,y)$ contains both $I_{P}(x)$ and  $X_{P}(y)$, as special case in the sense that
 $\mathcal{J}_{P}(x,1)=x^{n}I_{P}(\frac{1}{x})$ and $\mathcal{J}_{P}(1,y)=y^{n}X_{P}(\frac{1}{y})$, respectively. And they showed that $\mathcal{J}_{P}(x,y)$ is $S_{n}$-invariant by finding the relation between $\mathcal{J}_{P}(x,y)$ and $\widetilde{\mathcal{J}}_{P}(x,y)$, a polymatroid polynomial defined similarly as the corank-nullity definition of Tutte polynomial of matroids. They also proved that $\mathcal{L}'_{P}(x,y)=\frac{\mathcal{J}_{P}(x,y)}{x+y-1}$, where $\mathcal{L}'_{P}(x,y)$ is a two-variable invariant of a polymatroid $P$ defined by Cameron and Fink  in \cite{Cameron}.

In this paper, we first prove that $\mathcal{J}_{P}(x,t)$ and $\mathcal{J}_{P}(t,y)$ are interpolating for any real number $t\geq 1$, which extends interpolatory property of $I_{P}(x)$ and  $X_{P}(y)$ of a hypergraphical polymatroid $P$ in \cite{Guan}.
In \cite{Kalman1,Guan}, the coefficients of the linear and constant terms  in $I_{P}(x)$ and  $X_{P}(y)$ of a hypergraphical polymatroid $P$, have been studied. In this paper, as generalization, we shall study the coefficients of terms of degree $(n-1)$ and $n$ in $\mathcal{J}_{P}(x,1)$ and $\mathcal{J}_{P}(1,y)$ for a polymatroid $P$ with ground set $[n]$.

\section{Preliminaries}
\noindent

In this section, we will give some definitions and summarise some known results used in the next sections.

We first recall the definition of polymatroid given by Edmonds \cite{Edmonds}. Throughout the paper, we let  $2^{[n]}=\{I|I\subset [n]\}$, and $\textbf{e}_{1},\textbf{e}_{2},\ldots,\textbf{e}_{n}$ denotes the canonical basis of $\mathbb{R}^{n}$.
\begin{definition}  \cite{Edmonds} \label{def polymatroid}
Let  $f:2^{[n]}\rightarrow \mathbb{R}_{\geq 0}$ be a function, satisfying

\begin{enumerate}
\item[(1)] $f(\emptyset)=0$;

\item[(2)] $f(I)\leq f(J)$ for any $I\subset J\subset [n]$;

\item[(3)] $f(I)+f(J)\geq f(I\cup J)+f(I\cap J)$ for any $I,J\subset [n]$.
\end{enumerate}

\end{definition}
That is, $f$ is a monotone nondecreasing submodular function with $f(\emptyset)=0$. The pair $P=([n],f)$ is called a \emph{polymatroid}. Moreover, $\textbf{a}=(a_{1},a_{2},\cdots,a_{n})$ is called a \emph{basis} of the polymatroid $P$ with ground set $[n]$ and rank function $f$, if $\textbf{a}\in P$  and $\sum_{i\in [n]}a_{i}=f([n])$.
We denote the set of all bases in the polymatroid $P$ with $PB=\{\textbf{a}\in \mathbb{R}_{\geq 0}^{n}|\textbf{a}\in P \ \text{and} \ \sum_{i\in [n]}a_{i}=f([n])\}$.

In fact, Edmonds \cite{Edmonds} gave the following characterization: Polymatroids $P\subset \mathbb{R}^{n}$ correspond bijectively to submodular functions $f:2^{[n]}\rightarrow \mathbb{R}_{\geq 0}$ with $f(\emptyset)=0$. Let $f:2^{[n]}\rightarrow \mathbb{R}_{\geq 0}$ be a submodular function with $f(\emptyset)=0$, the associated polymatroid $P=P_{f}$ is given by
$$P=P_{f}=\left\{\textbf{x}\in \mathbb{R}_{\geq 0}^{n}\bigg|\sum_{i\in I}x_{i}\leq f(I) \ \text{for any subset}\ I\subset [n]\right\}.$$ Furthermore, the set of bases of the polymatroid $P$ is given by
$$PB=\left\{\textbf{x}\in \mathbb{R}_{\geq 0}^{n}\bigg|\sum_{i\in I}x_{i}\leq f(I) \ \text{for any subset}\ I\subset [n], \text{and} \sum_{i\in [n]}x_{i}=f([n])\right\}.$$ Conversely, suppose that $P$ is a polymatroid, then its rank function $f=f_{P}$ of the polymatroid $P$, defined by
$$f(I)=\max_{\textbf{a}\in P}\sum_{i\in I} a_{i},  \ \text{for any subset}\ I\subset [n],$$
is a submodular function.

In this paper, we only consider \emph{integer polymatroid}, and simply call them \emph{polymatroid}. In \cite{Herzog}, we know that an integer polymatroid $P$ on the ground set $[n]$ is a non-empty subset of $\mathbb{Z}_{\geq}^{n}$ satisfying the following axiom:

\emph{Basis Exchange Axiom}. For any $\textbf{a}=(a_{1},a_{2},\cdots,a_{n})$, $\textbf{b}=(b_{1},b_{2},\cdots,b_{n})$ and $\textbf{a}, \textbf{b}\in PB$, and any $i\in [n]$ such that $a_{i}>b_{i}$, there exists an index $j\in [n]$ such that $a_{j}<b_{j}$ and $\textbf{a}+\textbf{e}_{j}-\textbf{e}_{i}\in PB$, $\textbf{b}+\textbf{e}_{i}-\textbf{e}_{j}\in PB$.

It is easy to see that any matroid $M$ yields a polymatroid, and adding the same vector to each element (viewed as an $n$-dimensional vector of $\mathbb{R}^{n}$) of a polymatroid yields another polymatroid.

It is easy to verify that the set of all  hypertrees in a hypergraph $\mathcal{H}=(V,E)$ is the set of bases in a polymatroid with rank function $f:2^{E}\rightarrow \mathbb{Z}_{\geq}$, which is exactly the $\mu$-function defined by Kalman\cite{Kalman1} for any subset $E'$ of hyperedges $E$. In other words, polymatroids are an abstraction of hypergraphs. A polymatroid is called \emph{hypergraphical polymatroid} $P_{\mathcal{H}}$ if its bases are the set of all hypertrees for some hypergraph $\mathcal{H}$.

Moreover, integer polymatroids introduced by  Edmonds can be generalized to any integer, that is, a function $f:2^{[n]}\rightarrow \mathbb{Z}$ satisfies three conditions in Definition \ref{def polymatroid}.

Bernardi et al.  \cite{Bernardi} introduced  (internal and external) activity of a basis for  polymatroids and hence, the polymatroid Tutte polynomial as follows.
\begin{definition}\cite{Bernardi}
Let $P\subset \mathbb{Z}^{n}$ be a polymatroid  with  with ground set $[n]$ and rank function $f$.

For a basis $\textbf{a}\in PB$, an index $i\in [n]$ is \emph{internally active} if $\textbf{a}-\textbf{e}_{i}+\textbf{e}_{j}\notin PB$ for any $j<i$. Let $Int(\textbf{a})=Int_{P}(\textbf{a})\subset [n]$ denote the set of
internally active indices with respect to $\textbf{a}$ and $\iota(\textbf{a})=|Int(\textbf{a})|$.

For a basis $\textbf{a}\in PB$, an index $i\in [n]$ is \emph{externally active} if $\textbf{a}+\textbf{e}_{i}-\textbf{e}_{j}\notin PB$ for any $j<i$. Let $Ext(\textbf{a})=Ext_{P}(\textbf{a})\subset [n]$ denote the set of externally active indices with respect to $\textbf{a}$ and $\epsilon(\textbf{a})=|Ext(\textbf{a})|$.
\end{definition}

\begin{definition}\cite{Bernardi}
Let $P\subset \mathbb{Z}^{n}$ be a polymatroid  with ground set $[n]$ and rank function $f$. The \emph{polymatroid Tutte polynomial} $\mathcal{J}_{P}(x,y)$ is defined as
$$\mathcal{J}_{P}(x,y)=\sum_{\textbf{a}\in PB}x^{oi(\textbf{a})}y^{oe(\textbf{a})}(x+y-1)^{ie(\textbf{a})}$$
where $$oi(\textbf{a})=|Int(\textbf{a})\setminus Ext(\textbf{a})|,$$ $$oe(\textbf{a})=|Ext(\textbf{a})\setminus Int(\textbf{a})|,$$ $$ie(\textbf{a})=|Int(\textbf{a})\cap Ext(\textbf{a})|.$$
\end{definition}
Note that the smallest index $i=1$ must be simultaneously internally and externally active  for any basis of any polymatroid. Then we have that the polynomial $\mathcal{J}_{P}(x,y)$ is always divisible by $(x+y-1)$.
It is easy to verify that if $P_{\mathcal{H}}$ is a hypergraphical polymatroid with ground set $[n]$, then the polynomial $\mathcal{J}_{P_{\mathcal{H}}}(x,y)$ is a generalization of both the interior polynomial $I_{P_{\mathcal{H}}}(x)$ and the exterior polynomial $X_{P_{\mathcal{H}}}(y)$, introduced by K\'{a}lm\'{a}n in \cite{Kalman1}. In fact, we have,
 $\mathcal{J}_{P_{\mathcal{H}}}(x,1)=x^{n}I_{P_{\mathcal{H}}}(\frac{1}{x})$ and $\mathcal{J}_{P_{\mathcal{H}}}(1,y)=y^{n}X_{P_{\mathcal{H}}}(\frac{1}{y})$, respectively.

In \cite{Bernardi}, Bernardi et al. proved that the polynomial $\mathcal{J}_{P}(x,y)$ is translation invariant and $S_{n}$-invariant. And they also found that the polymatroid Tutte polynomial can reduce to the classical Tutte polynomial $T_{M}(x,y)$ of a matroid $M$ as follows.
If $M\subset 2^{[n]}$ is a matroid of rank $d$ on the ground set $[n]$, and  $P=P(M)\subset \{0,1\}^{n}$ is the corresponding polymatroid, then $\mathcal{J}_{P}(x,y)=x^{n-d}y^{d}T_{M}(\frac{x+y-1}{y},\frac{x+y-1}{x})$.

The \emph{dual} polymatroid  of  a polymatroid $P$ is denoted by $-P$, whose the set of its bases
$$-PB:=\{(-a_{1},\ldots,-a_{n})|(a_{1},\ldots,a_{n})\in PB\}.$$

It is easy to see the relation between the Tutte polynomial of a polymatroid and its dual as follows.
\begin{theorem}\cite{Bernardi}\label{thm dual}
Let $P\subset \mathbb{Z}^{n}$ be a polymatroid and $-P$ be its dual polymatroid.  Then $\mathcal{J}_{P}(x,y)=\mathcal{J}_{-P}(y,x)$.
\end{theorem}

In the end of this section, we will introduce the definition of an interpolating polynomial.
\begin{definition}
The \emph{support} of a polynomial $f(x)=\sum\limits_{i=0}^{m}a_{i}x^{i}$ is the set $supp(f)=\{i|a_{i}\neq 0\}$ of  indices of the non-zero coefficients. The polynomial $f$ is called \emph{interpolating} if its $supp(f)$ is an integer interval $[n_1, n_2]$
of all integers from $n_1$ to $n_2$, inclusive.
\end{definition}

\section{Interpolating behavior of $\mathcal{J}_{P}(x,t)$ and $\mathcal{J}_{P}(t,y)$}
\noindent

In this section, we will prove that both $\mathcal{J}_{P}(x,t)$ and $\mathcal{J}_{P}(t,y)$ are interpolating for any fixed real number $t\geq 1$. By the duality theorem, we only consider  $\mathcal{J}_{P}(x,t)$.  We first  prove the special case $t=1$, then prove the general case.

Let $P\subset \mathbb{Z}^{n}$ be a polymatroid  with ground set $[n]$ and rank function $f$. For a basis $\textbf{a}\in PB$, let
$$\mathcal{I}(\textbf{a})=\mathcal{I}_{P}(\textbf{a})=\left\{I\subset [n]\bigg| \sum_{i\in I}a_{i}= f(I)\right\}$$
be the set of tight set for $\textbf{a}$.

The next Theorem holds from Theorem 44.2 in \cite{Schrijver} since $f$ is submodular.

\begin{lemma}\label{tight}
Let $P\subset \mathbb{Z}^{n}$ be a polymatroid  with ground set $[n]$ and rank function $f$. For any basis $\textbf{a}\in PB$, if $I,J\in \mathcal{I}(\textbf{a})$, then $I\cup J,I\cap J\in \mathcal{I}(\textbf{a})$.
\end{lemma}

\begin{lemma}\label{transfer}
Let $P\subset \mathbb{Z}^{n}$ be a polymatroid  with ground set $[n]$ and rank function $f$, $\textbf{a}\in PB$, $i,j\in [n]$ and $i\neq j$. Then $\textbf{b}=\textbf{a}+\textbf{e}_{i}-\textbf{e}_{j}\in PB$  if and only if $I\notin \mathcal{I}(\textbf{a})$ for any $I$ satisfying $i\in I\subset [n]\setminus j$.
\end{lemma}

\begin{proof}
For sufficiency,  it is easy to see that $\sum_{i\in [n]}b_{i}=\sum_{i\in [n]}a_{i}=f([n])$. We divide three cases to discuss for any $I\subset [n]$.
\begin{enumerate}
\item[(1)] If $i,j\in I$ or $i,j\notin I$, then $$\sum_{i'\in I}b_{i'}=\sum_{i'\in I}a_{i'}\leq f(I).$$
\item[(2)] If $j\in I$ or $i\notin I$, then  $$\sum_{i'\in I}b_{i'}=(\sum_{i'\in I}a_{i'})-1< f(I).$$
\item[(3)] If $i\in I$ or $j\notin I$, then by known conditions, we have $$\sum_{i'\in I}b_{i'}=(\sum_{i'\in I}a_{i'})+1\leq f(I).$$
\end{enumerate}
Hence, $\sum_{i'\in I}b_{i'}\leq f(I)$ for any $I\subset [n]$, that is, $\textbf{b}=\textbf{a}+\textbf{e}_{i}-\textbf{e}_{j}\in PB$.

We prove the necessity by contradiction. Assume that there is a subset $I\subset [n]\setminus j$ and $i\in I$ such that $I\in \mathcal{I}(\textbf{a})$, that is, $\sum_{i'\in I}a_{i'}= f(I)$. Then $\sum_{i'\in I}b_{i'}=\sum_{i'\in I}a_{i'}+1> f(I)$ which contradicts the fact that $\textbf{b}\in PB$.
\end{proof}

By Lemma \ref{transfer}, we can obtain the necessary and sufficient condition that an index is internally or externally active for a basis of  polymatroids, which has been obtained in Lemma 4.2 in \cite{Bernardi}.

\begin{lemma} \cite{Bernardi}\label{ia}
Let $P\subset  \mathbb{Z}^{n}$ be a polymatroid  with ground set $[n]$ and rank function $f$. Let $\textbf{a}\in PB$ be a basis of  the polymatroid $P$. Then

\begin{itemize}
  \item [(1)] An index $i\in [n]$ is internally active with respect to $\textbf{a}$,
if and only if there exists a subset $I\subset [n]$ such that $i=min(I)$ and $[n]\setminus I\in \mathcal{I}(\textbf{a})$.
  \item [(2)] An index $i\in [n]$ is externally active with respect to $\textbf{a}$,
if and only if there exists a subset $I\subset [n]$ such that $i=min(I)$ and $I\in \mathcal{I}(\textbf{a})$.
\end{itemize}

\end{lemma}

Now we let $P\subset \mathbb{Z}^{n}$ be a polymatroid  with ground set $[n]$ and rank function $f$, and  indices $i,j\in [n]$ and $i\neq j$. Let $\textbf{a},\textbf{b}$ be the bases of the polymatroid $P$, where $a_{t}=b_{t}$ for any $t\in [n]\setminus \{i,j\}$ and $a_{i}<b_{i}$. Then we have the following lemma.

\begin{lemma} \label{a to b}
 For any index $k>i$, if $k\in Int(\textbf{a})$, then  $k\in Int(\textbf{b})$.
\end{lemma}

\begin{proof}
If $k\in Int(\textbf{a})$, then there exists a subset $I\subset [n]$ such that $k=min(I)$ and $[n]\setminus I\in \mathcal{I}(\textbf{a})$ by Lemma \ref{ia} (1). It is clear that $i\in [n]\setminus I$ since $k>i$. Note that $f([n]\setminus I)\geq\sum_{i'\in [n]\setminus I}b_{i'}\geq\sum_{i'\in [n]\setminus I}a_{i'}=f([n]\setminus I)$. We have that $f([n]\setminus I)=\sum_{i'\in [n]\setminus I}b_{i'}$, that is, $[n]\setminus I\in \mathcal{I}(\textbf{b})$. Hence, $k\in Int(\textbf{b})$.
\end{proof}

\begin{corollary}\label{a and b}
Let $P\subset \mathbb{Z}^{n}$ be a polymatroid  with ground set $[n]$ and rank function $f$. Suppose that indices $i,j\in [n]$ and $i\neq j$, $\textbf{a},\textbf{b}\in PB$, $\textbf{a}\neq\textbf{b}$  and $a_{t}=b_{t}$ for any $t\in [n]\setminus \{i,j\}$. For any index $k>max\{i,j\}$, we have $k\in Int(\textbf{a})$ if and only if  $k\in Int(\textbf{b})$.
\end{corollary}
\begin{proof}
It follows directly from Lemma \ref{a to b}.
\end{proof}

Now we are in a position to prove the following key lemma of this section.

\begin{lemma}\label{lem main}
Let $P\subset \mathbb{Z}^{n}$ be a polymatroid  with ground set $[n]$ and rank function $f$. If there exists a basis of $P$ with internal activity $m$, then there also exists a basis of $P$ whose internal activity is $m+1$ for any integer $m<n$.
\end{lemma}

\begin{proof}
For an integer $m<n$, let
$$A_{m}=\{\textbf{a}'\in PB| \iota(\textbf{a}')=m\},$$
$$B_{i}=\{\textbf{b}'\in A_{m}|i \ \text{is the smallest index which is not internally active with respect to} \ \textbf{b}'\}.$$

Note $A_{m}\neq\emptyset$ and for any $\textbf{a}\in PB$, we have that $1\in Int(\textbf{a})$. Then there is an index $i\in [n]$ such that $B_{i}\neq\emptyset$, and $\bigcup_{i=2}^{n}B_{i}=A_{m}$ and $B_{i}\cap B_{j}=\emptyset$ for any $i\neq j$.

Let $\textbf{a}\in B_{i}$ be a basis of the polymatroid $P$ whose the entry $a_{i}$ of $\textbf{a}$ is minimum and $$j=\min\{i'|\textbf{a}-\textbf{e}_{i}+\textbf{e}_{i'}\in PB \ \text{and} \ i'<i\}.$$

 Let $\textbf{b}=\textbf{a}-\textbf{e}_{i}+\textbf{e}_{j}$.
Then we claim that $k$ is internally active with respect to $\textbf{a}$ if and only if $k$ is internally active with respect to $\textbf{b}$ for any index $k\neq i$.

For $k\leq j$, there exists a subset $I\subset [n]$ such that $k=min(I)$ and $[n]\setminus I\in \mathcal{I}(\textbf{a})$ by Lemma \ref{ia} (1) since $k$ is internally active with respect to $\textbf{a}$. For any $t<k\leq j$, by Lemma \ref{transfer}, there exists a subset $I_{t}\subset [n]\setminus \{i\}$ and $t\in I_{t}$ such that $I_{t}\in \mathcal{I}(\textbf{a})$ since $\textbf{a}-\textbf{e}_{i}+\textbf{e}_{t}\notin PB$. We have that $([n]\setminus I)\cap (\bigcup_{t=1}^{k-1}I_{t})\in \mathcal{I}(\textbf{a})$ by Lemma \ref{tight}. Let $([n]\setminus I)\cap (\bigcup_{t=1}^{k-1}I_{t})=[n]\setminus I'$.  Then $[n]\setminus I'\in \mathcal{I}(\textbf{a})$, $k=min(I')$ and $i\notin [n]\setminus I'$. We claim that $[n]\setminus I'\in \mathcal{I}(\textbf{b})$ since $\sum_{i\in [n]\setminus I'}b_{i}\geq \sum_{i\in [n]\setminus I'}a_{i}=f([n]\setminus I')$ (In fact, $j\notin [n]\setminus I'$ and $\sum_{i\in [n]\setminus I'}b_{i}=\sum_{i\in [n]\setminus I'}a_{i}=f([n]\setminus I')$). Hence, the index $k$ is internally active with respect to $\textbf{b}$ by Lemma \ref{ia}.

For $j<k<i$, we have that $k$ is internally active with respect to $\textbf{a}$. Then $k$ is also internally active with respect to $\textbf{b}$ by Lemma \ref{a to b}.

For $k>i$, the claim is true by Corollary \ref{a and b}.

It is obvious that $i$ is internally active with respect to $\textbf{b}$ by the choice of $\textbf{a}$. In fact, if $i$ is not internally active with respect to $\textbf{b}$, then $\textbf{b}\in B_{i}$.  Since $a_{i}>b_{i}$, it contradicts the choice of $\textbf{a}$. Thus, $\iota(\textbf{b})=m+1$.
\end{proof}

\begin{theorem}\label{thm main1}
Let $P\subset \mathbb{Z}^{n}$ be a polymatroid  with ground set $[n]$ and rank function $f$. Then $\mathcal{J}_{P}(x,1)$ and $\mathcal{J}_{P}(1,y)$ are interpolating.
\end{theorem}
\begin{proof}
It follows directly from Lemma \ref{lem main} and Theorem \ref{thm dual} since $\mathcal{J}_{P}(x,1)=\sum_{\textbf{a}\in PB}x^{oi(\textbf{a})}x^{ie(\textbf{a})}=\sum_{\textbf{a}\in PB}x^{\iota(\textbf{a})}$.
\end{proof}

In fact, Theorem \ref{thm main1} can be generalized to any real number not less than 1.
\begin{theorem} \label{main}
Let $P\subset \mathbb{Z}^{n}$ be a polymatroid  with ground set $[n]$ and rank function $f$. Then both $\mathcal{J}_{P}(x,t)$ and $\mathcal{J}_{P}(t,y)$ are interpolating for any fixed real number $t\geq 1$.
\end{theorem}
\begin{proof}
The conclusion holds for $t=1$ by Theorem \ref{thm main1}. Now we consider the case $t>1$.

By the definition of polymatroid Tutte polynomial $\mathcal{J}_{P}(x,y)$, we have that
\begin{eqnarray*}
\mathcal{J}_{P}(x,t)&=&\sum_{\textbf{a}\in PB}x^{oi(\textbf{a})}t^{oe(\textbf{a})}(x+t-1)^{ie(\textbf{a})}\\
&=&\sum_{\textbf{a}\in PB}t^{oe(\textbf{a})}x^{oi(\textbf{a})}[x+(t-1)]^{ie(\textbf{a})}\\
&=&\sum_{\textbf{a}\in PB}t^{oe(\textbf{a})}x^{oi(\textbf{a})}\left[\sum_{i=0}^{ie(\textbf{a})}\binom{ie(\textbf{a})}{i}x^{i}(t-1)^{ie(\textbf{a})-i}\right]\\
&=&\sum_{\textbf{a}\in PB}t^{oe(\textbf{a})}\left[\sum_{i=0}^{ie(\textbf{a})}\binom{ie(\textbf{a})}{i}(t-1)^{ie(\textbf{a})-i}x^{oi(\textbf{a})+i}\right].
\end{eqnarray*}

For any $\textbf{a}\in PB$, $\textbf{a}$ contributes the lowest degree $oi(\textbf{a})$ and the highest degree $oi(\textbf{a})+ie(\textbf{a})=\iota(\textbf{a})$ to  $\mathcal{J}_{P}(x,t)$, and the coefficient of every contributing degree from $oi(\textbf{a})$ to $\iota(\textbf{a})$ is positive.

Let $\textbf{a}\in PB$ be the basis of $P$ such that $oi(\textbf{a})$ is minimum. Then the lowest degree in $\mathcal{J}_{P}(x,t)$ is $oi(\textbf{a})$. Based on the above discussion, the coefficients of the degrees from $oi(\textbf{a})$ to $\iota(\textbf{a})$ are positive. The coefficients  of the degrees from $\iota(\textbf{a})+1$ to $n$ are positive  by Lemma \ref{lem main}. Hence, the conclusion is true.
\end{proof}

\section{High-order terms of $\mathcal{J}_{P}(x,1)$ and $\mathcal{J}_{P}(1,y)$}
\noindent

In this section, we will study the coefficients of the terms of degree $n$ and $(n-1)$ in $\mathcal{J}_{P}(x,1)$ and $\mathcal{J}_{P}(1,y)$.

Let $P\subset \mathbb{Z}^{n}$ be a polymatroid  with ground set $[n]$ and rank function $f$. And let $\textbf{a}^{max}\in PB$ be the lexicographically maximal basis of the polymatroid $P$, that is, if $\textbf{b}\neq \textbf{a}^{max}$ is another  basis of the polymatroid $P$ and $j=\min\{i'|b_{i'}\neq a^{max}_{i'}\}$, then $b_{j}<a^{max}_{j}$. It is easy to see that $\iota(\textbf{a}^{max})=n$, and for any $k\in [n]$, we have that $\sum_{i\in [k]}a^{max}_{i}=f([k])$, that is, $a^{max}_{k}=f([k])-f([k-1])$.

For any basis $\textbf{b}\neq\textbf{a}^{max}$, let $j=\max\{i'|b_{i'}\neq a^{max}_{i'}\}$, then $b_{j}>a^{max}_{j}$ since $\sum_{i\in [n]\setminus [j]}b_{i}=\sum_{i\in [n]\setminus [j]}a^{max}_{i}$, $\sum_{i\in [n]}b_{i}=\sum_{i\in [n]}a^{max}_{i}$ and $\sum_{i\in [j-1]}b_{i}<f([j-1])=\sum_{i\in [j-1]}a^{max}_{i}$. By Basis Exchange Axiom, we have that there is an index $i\in [n]$ with $b_{i}<a^{max}_{i}$ such that $\textbf{b}+\textbf{e}_{i}-\textbf{e}_{j}\in PB$. It is easy to see that $i<j$. Then $j\notin Int(\textbf{b})$, that is, $\iota(\textbf{b})<n$. We have the following conclusion by above discussion and the duality.
\begin{theorem} \label{n}
Let $P\subset \mathbb{Z}^{n}$ be a polymatroid  with ground set $[n]$ and rank function $f$. Then the coefficients of the terms of degree $n$ in $\mathcal{J}_{P}(x,1)$ and $\mathcal{J}_{P}(1,y)$ are 1.
\end{theorem}

For a hypergraphical polymatroid $P_{\mathcal{H}}$, we know that $\mathcal{J}_{P_{\mathcal{H}}}(x,1)=x^{n}I_{P_{\mathcal{H}}}(\frac{1}{x})$ and $\mathcal{J}_{P_{\mathcal{H}}}(1,y)=y^{n}X_{P_{\mathcal{H}}}(\frac{1}{y})$, respectively. In \cite{Kalman1}, K\'{a}lm\'{a}n obtained the coefficients of the constant term in the interior and exterior polynomials.

Next, we study the coefficients of the terms of degree $(n-1)$ in $\mathcal{J}_{P}(x,1)$ and $\mathcal{J}_{P}(1,y)$.
\begin{theorem} \label{main}
Let $P\subset \mathbb{Z}^{n}$ be a polymatroid   with ground set $[n]$ and rank function $f$. Then the coefficients of the terms of degree $(n-1)$ in $\mathcal{J}_{P}(x,1)$ and $\mathcal{J}_{P}(1,y)$ are $\sum_{i\in [n]}f(\{i\})-f([n])$ and $\sum_{i\in [n]}f([n]\setminus \{i\})-(n-1)f([n])$, respectively.
\end{theorem}
\begin{proof}
We firstly show that the coefficient of the degree $(n-1)$  in $\mathcal{J}_{P}(x,1)$ is $\sum_{i\in [n]}f(\{i\})-f([n])$.

We now show that for any $i\in [n]$ and $0\leq t \leq f(\{i\})-a^{max}_{i}$, there is a basis $\textbf{b}$ of the polymatroid  $P$ such that $b_{k}=a^{max}_{k}$ for any $k>i$ and $b_{i}=a^{max}_{i}+t$, where $\textbf{a}^{max}\in PB$ is the lexicographically maximal basis of the polymatroid $P$. Suppose, to the contrary, that for some $i\in [n]$ and $0\leq t \leq f(\{i\})-a^{max}_{i}$, there does not exist such basis. Assume, without loss of generality, the conclusion is true for $t=s$ and not true for $t=s+1$, where $0\leq s < f(\{i\})-a^{max}_{i}$,  since it holds for $t=0$. Let $\textbf{a}$ be a basis of $P$ with $a_{k}=a^{max}_{k}$ for any $k>i$ and $a_{i}=a^{max}_{i}+s$. It is clear that $a_{i}<f(\{i\})$ and $i\in Ext(\textbf{a})$.
Then  by Lemma \ref{ia} (2), we have that there exists a subset $I\subset [n]$ such that $i=min (I)$ and $I\in \mathcal{I}(\textbf{a})$. It is clear that $[i]\in \mathcal{I}(\textbf{a})$ by the construction of $\textbf{a}$. Then $\{i\}=[i]\cap I\in \mathcal{I}(\textbf{a})$ by Lemma \ref{tight}, that is, $a_{i}=f(\{i\})$, a  contradiction.

For any $i\in [n]$ and $0\leq t \leq f(\{i\})-a^{max}_{i}$, let $\textbf{a}^{it}$ $(1\leq t \leq f(\{i\})-a^{max}_{i})$ be  reverse lexicographically maximal basis among these, that is, if $\textbf{b}\neq \textbf{a}^{it}$ is another such basis  and $j=\max\{i'|b_{i'}\neq a^{it}_{i'}\}$, then $b_{j}>a^{it}_{j}$. Since $\sum_{i\in [n]}[f(\{i\})-a^{max}_{i}]=\sum_{i\in [n]}f(\{i\})-\sum_{i\in [n]}a^{max}_{i}=\sum_{i\in [n]}f(\{i\})-\sum_{i\in [n]}[f([i])-f([i-1])]=\sum_{i\in [n]}f(\{i\})-f([n])$, there are $\sum_{i\in [n]}f(\{i\})-f([n])$ those bases.

We next show that $Int(\textbf{a}^{it})=[n]\setminus \{i\}$, that is, $\iota(\textbf{a}^{it})=n-1$.

\begin{enumerate}
\item[(1)] If $k>i$, then $k\in Int(\textbf{a}^{it})$ since $\sum_{i'\in [k-1]}a_{i'}^{it}=\sum_{i'\in [k-1]}a_{i'}^{max}=f([k-1])$.
%
%$[k-1]\in \mathcal{I}(\textbf{a}^{it})$.
\item[(2)] If $k<i$, then $k\in Int(\textbf{a}^{it})$. Otherwise, there is a basis $\textbf{c}\in PB$  with $c_{k}=a^{it}_{k}-1$, $c_{j}=a^{it}_{j}+1$ for some $j<k$ and $c_{i'}=a^{it}_{i'}$ for any $i'\in [n]\setminus \{k,j\}$. It contradicts the choice of $\textbf{a}^{it}$.
\item[(3)] By Basis Exchange Axiom (similar to the proof in Theorem \ref{n}), we have that $i\notin Int(\textbf{a}^{it})$.
\end{enumerate}

In the following, we show that if $\textbf{a}\in PB$ is a basis with $Int(\textbf{a})=[n]\setminus \{i\}$, then $\textbf{a}$ is one of the $\textbf{a}^{it}$.
%such that $i$ is not unique internally active index

For any $k>i$ and $t\geq k$, we have that $t\in Int(\textbf{a})$. Then by Lemma \ref{ia} (1), we have that there exists a subset $I_{t}\subset [n]$ such that $t=min (I_{t})$ and $[n]\setminus I_{t}\in \mathcal{I}(\textbf{a})$. Then $\cap_{t\geq k}([n]\setminus I_{t})=[k-1]\in \mathcal{I}(\textbf{a})$ by Lemma \ref{tight}. It implies that $a_{k}=a^{max}_{k}$ for any $k>i$.

Assume that $t=a_{i}-a^{max}_{i}$. Then $t\geq 1$, otherwise, $i\in Int(\textbf{a})$. We have that $a_{k}=a^{it}_{k}$ for any $k\geq i$ since $a^{it}_{k}=a^{max}_{k}$ for any $k>i$. We now claim that $\textbf{a}=\textbf{a}^{it}$. By contradiction, suppose that $\textbf{a}\neq\textbf{a}^{it}$ and $j=max\{i'|a_{i'}\neq a^{it}_{i'}\}$. Then $j<i$ and $a_{j}>a^{it}_{j}$ by the choice of $\textbf{a}^{it}$.
By Lemma \ref{ia} (1), there exists a subset $I\subset [n]$ such that $j=min(I)$ and $[n]\setminus I\in \mathcal{I}(\textbf{a})$ since $j\in Int(\textbf{a})$. Note that $[j-1]\subset [n]\setminus I$, $\sum_{i'\in [n]}a^{it}_{i'}=\sum_{i'\in [n]}a_{i'}=f([n])$ and $a_{i'}=a^{it}_{i'}$ for any $i'>j$.  We have that $([n]\setminus I)\cap [j-1]=[j-1]$ and $\sum_{i'\in [j-1]}a^{it}_{i'}>\sum_{i'\in  [j-1]}a_{i'}$. Then $\sum_{i'\in [n]\setminus I}a^{it}_{i'}=\sum_{i'\in  [j-1]}a^{it}_{i'}+\sum_{i'\in ([n]\setminus I) \setminus [j-1]}a^{it}_{i'}>\sum_{i'\in  [j-1]}a_{i'}+\sum_{i'\in ([n]\setminus I)\setminus [j-1]}a_{i'}=\sum_{i'\in [n]\setminus I}a_{i'}=f([n]\setminus I)$. It contradicts the fact $\textbf{a}^{it}\in PB$. Then $\textbf{a}=\textbf{a}^{it}$. Hence, the coefficient of the degree $(n-1)$ in $\mathcal{J}_{P}(x,1)$ is $\sum_{i\in [n]}f(\{i\})-f([n])$.

Now we consider the coefficient of the degree $(n-1)$ in  $\mathcal{J}_{P}(1,y)$. We have that $\mathcal{J}_{P}(1,y)=\mathcal{J}_{-P}(y,1)$ by Theorem \ref{thm dual}.
In particular, the coefficient of the degree $(n-1)$ in $\mathcal{J}_{P}(1,y)$ is equal to the coefficient of the degree $(n-1)$ in $\mathcal{J}_{-P}(y,1)$.

Let $-P\subset Z^{n}$ be a polymatroid with rank function $f'$. Then by definition of the dual polymatroid, we have that
 $f'(\{i\})=\max_{\textbf{-a}\in -PB}-a_{i}=-\min_{\textbf{a}\in PB}a_{i}=-(f([n])-f([n]\setminus \{i\}))$ and $f'([n])=-f([n])$.
 Hence $\sum_{i\in [n]}f'(\{i\})-f'([n])=-\sum_{i\in [n]}(f([n])-f([n]\setminus \{i\}))+f([n])=\sum_{i\in [n]}f([n]\setminus \{i\})-(n-1)f([n])$. The conclusion holds.
\end{proof}

In \cite{Kalman1}, K\'{a}lm\'{a}n studied the coefficient of the linear term in the interior polynomial. In \cite{Guan}, the authors studied the coefficient of the linear term in the exterior polynomial. Their results can be derived form Theorem \ref{main} as follows.

\begin{corollary} \cite{Kalman1}
Let $\mathcal{H}=(V,E)$ be a connected hypergraph and $Bip \mathcal{H}=(V\cup E, \varepsilon)$ be the bipartite graph associated to  $\mathcal{H}$. Then the coefficient of the linear term in  the interior polynomial $I_{\mathcal{H}}(x)$ is $|\varepsilon|-(|E|+|V|)+1$.
\end{corollary}
\begin{proof}
Let $P_{\mathcal{H}}$ be a hypergraphical polymatroid, with rank function $f$,  corresponding to the hypergraph $\mathcal{H}$. It is clear that $f(E)=|V|-1$, and $f(e)=d_{Bip \mathcal{H}}(e)-1$ for any $e\in E$, that is, $\sum_{e\in E}f(e)=|\varepsilon|-|E|$. Then $\sum_{e\in E}f(e)-f(E)=|\varepsilon|-(|E|+|V|)+1$.
\end{proof}

\begin{corollary} \cite{Guan}
Let $\mathcal{H}=(V,E)$ be a connected hypergraph. If  $\mathcal{H}-e$ is connected for each $e\in E$, then the coefficient of the linear term in  the exterior polynomial $X_{\mathcal{H}}(y)$ is $|V|-1$.
\end{corollary}
\begin{proof}
Let $P_{\mathcal{H}}$ be a hypergraphical polymatroid, with rank function $f$,  corresponding to the hypergraph $\mathcal{H}$. It is clear that $f(E)=|V|-1$ and $n=|E|$. Let $Bip \mathcal{H}$ be the bipartite graph associated to  $\mathcal{H}$. If $\mathcal{H}-e$ is connected for each $e\in E$, then there exists a spanning tree $\tau$ of $Bip \mathcal{H}$ with the degree of $e$ in $\tau$, $d_{\tau}(e)=1$, that is, $f(E\setminus \{e\})=|V|-1$. It implies that $\sum_{e\in E} f(E\setminus \{e\})=|E|(|V|-1)$. Then $\sum_{e\in E} f(E\setminus \{e\})-(|E|-1)f(E)=|V|-1$.
\end{proof}

\section{Concluding remarks}
\noindent

Let $P\subset \mathbb{Z}_{\geq}^{3}$ be a polymatroid. The rank function $f:2^{[n]}\rightarrow \mathbb{Z}_{\geq}$ of the polymatroid $P$ is given by $f(\{1,2,3\})=f(\{1,2\})=f(\{1,3\})=f(\{2,3\})=3$, $f(\{1\})=f(\{2\})=2$, $f(\{3\})=1$, and $f(\emptyset)=0$. Then $\mathcal{J}_{P}(x,y)=(x+y-1)(y^{2}+y+2xy+x^{2})$. We have that $\mathcal{J}_{P}(x,\frac{1}{3})=x^{3}-\frac{8}{27}$,  $\mathcal{J}_{P}(0,y)=y^{3}-y$, and $\mathcal{J}_{P}(\frac{1+ \sqrt{13}}{6},y)=y^{3}+\frac{1+ \sqrt{13}}{2}y^{2}+\frac{-11+ \sqrt{13}}{54}$, that is, they are not interpolating. In this paper, we have showed  that $\mathcal{J}_{P}(x,t)$ and $\mathcal{J}_{P}(t,y)$ are  interpolating for any real number $t\geq 1$ (Theorem \ref{main}).  A natural open problem is to further explore the interpolating behavior of $\mathcal{J}_{P}(x,t)$ and $\mathcal{J}_{P}(t,y)$ in the case that $t<1$.

\section*{Acknowledgements}
\noindent
This work is supported by National Natural Science Foundation of China (No. 12171402).

\section*{References}
\bibliographystyle{model1b-num-names}
\bibliography{<your-bib-database>}

\end{document}